\documentclass[envcountsect]{llncs}
\usepackage[utf8]{inputenc}
\usepackage{amssymb,amsmath,mathrsfs}
\usepackage[english]{babel}
\usepackage[unicode,unicode=true,bookmarks=false]{hyperref}
\usepackage[shortlabels]{enumitem}
\usepackage{mathtools}
\usepackage[%
backend=biber,
bibencoding=utf8,
sorting=none,
style=numeric,
language=autobib,
autolang=other,
clearlang=true,
defernumbers=true,
sortcites=true,
doi=true,
isbn=true,
]{biblatex}
\usepackage{tikz}
\usetikzlibrary{automata, arrows.meta, positioning}
\usepackage{csquotes}
\usepackage{xurl}
\hypersetup{breaklinks=true}

\usepackage{environ}

\newcommand{\repeattheorem}[1]{%
	\begingroup
	\renewcommand{\thetheorem}{\ref{#1}}%
	\expandafter\expandafter\expandafter\theorem
	\csname reptheorem@#1\endcsname
	\endtheorem
	\endgroup
}

\NewEnviron{reptheorem}[1]{%
	\global\expandafter\xdef\csname reptheorem@#1\endcsname{%
		\unexpanded\expandafter{\BODY}%
	}%
	\expandafter\theorem\BODY\unskip\label{#1}\endtheorem
}

\newcommand{\NN}{\mathbb{N}}
\newcommand{\ZZ}{\mathbb{Z}}
\newcommand{\QQ}{\mathbb{Q}}

\newcommand{\PrA}{\mathop{\mathbf{PrA}}\nolimits}
\newcommand{\PA}{\mathop{\mathbf{PA}}\nolimits}
\newcommand{\BA}{\mathop{\mathbf{BA}}\nolimits}

\newcommand{\Th}{\mathop{\mathbf{Th}}\nolimits}

\makeatletter

\renewcommand{\epsilon}{\varepsilon}
\renewcommand{\phi}{\varphi}
\newcommand{\sref}[2]{\hyperref[#2]{#1 \ref*{#2}}}
\newcommand{\dref}[2]{\hyperref[#2]{ #1 }}

\newcommand{\Ac}{\mathcal{A}}

\newcommand{\eqdef}{\stackrel{\mbox{\tiny\rm def}}{=}}

\newcommand{\ra}{\rightarrow}
\newcommand{\Ra}{\Rightarrow}

\newcommand{\Lra}{\Leftrightarrow}
\newcommand{\lra}{\leftrightarrow}

\spnewtheorem{hyp}{Conjecture}[section]{\bfseries}{\itshape}
\spnewtheorem{ex}{Example}{\bfseries}{\itshape}
\spnewtheorem{stm}{Statement}[section]{\bfseries}{\itshape}

\makeatletter
\newcommand{\dotminus}{\mathbin{\text{\@dotminus}}}

\newcommand{\@dotminus}{%
	\ooalign{\hidewidth\raise1ex\hbox{.}\hidewidth\cr$\m@th-$\cr}%
}
\makeatother

\makeatletter
\let\blx@rerun@biber\relax
\makeatother

\begin{document}
	\author{Alexander Zapryagaev\thanks{Work of Alexander Zapryagaev is supported by the Russian Science Foundation under grant 21-11-00318.}}
	\title{On Non-Standard Models of B\"uchi Arithmetics}
	\institute{Steklov Mathematical Institute of Russian Academy of Sciences, 8, Gubkina Str., Moscow, 119991, Russian Federation}
	
	\maketitle
	
	\begin{abstract}
		  B\"uchi arithmetics $\BA_n$, $n\ge 2$, are extensions of Presburger arithmetic with an unary functional symbol $V_n(x)$ denoting the largest power of $n$ that divides $x$. We explore the structure of non-standard models of B\"uchi arithmetics and construct an example of a countable non-standard model of $\BA_2$.
	\end{abstract}
	
	\section{Preliminaries}
	
	\begin{definition}
			A \textbf{B\"uchi arithmetic} $\BA_n$ is the theory $\Th(\NN;=,+,V_n)$, $V_n$ is an unary functional symbol such that $V_n(x)$ is the largest power of $n$ that divides $x$ (we set $V_n(0):=0$ by definition).
		\end{definition}
	
        The theories $\BA_n$, $n\ge 2$, are conservative extensions of Presburger arithmetic $\PrA=\Th(\NN;=,+)$. These theories were proposed by R.~B\"uchi in order to describe the recognizability of sets of natural numbers by finite automata through definability in some arithmetic language. Namely, automatic structures are exactly those one-dimensionally, not necessarily with absolute equality, interpretable in $(\NN;=,+,V_n)$. This follows from the definition of an automatic structure \cite[Definition 1.4]{kn} and B\"uchi-Bruy\`ere Theorem \cite{bruyere,bv}.
    
    As shown in the author's dissertation \cite{zapthe} (Theorem 4.3.4), the theories $\BA_n$ are actually mutually interpretable for distinct $n\ge 2$:
    
    \begin{theorem}
		Each $\BA_k$ is interpretable in any of $\BA_l$, $k,l\ge 2$.
    \end{theorem}

    The author has also established \cite{zap2022}:

    \begin{theorem}\label{mthe}
		Let $\iota$ be a (one- or multi-dimensional) interpretation of $\BA_n$ in $(\NN;=,+,V_n)$. Then the internal model induced by $\iota$ is isomorphic to the standard one.
    \end{theorem}

    Note that this partially answers a question by A.~Visser \cite{vf}.

    The claim of \sref{Theorem}{mthe} could be restated in the following way:

    \begin{corollary}
        There are no automatic non-standard models of $\BA_n$. 
    \end{corollary}
    \begin{proof}
        If a non-standard model $\Ac$ of $\BA_n$ was automatic, there would exist a one-dimensional interpretation of $\Ac$ in $(\NN;=,+,V_n)$.
    \end{proof}

    In this article, we explore the structure of non-standard models of $\BA_n$ and aim to construct an explicit example of a countable non-standard model of $\BA_2$.
    
    \section{Basic axioms}

    It is well-known that $\PrA$ in the extended language $\{=,+,<,0,1,\{\equiv_n\}_{n=2}^\infty\}$ has an equivalent axiomatic definition, as the first-order theory given by the following recursive set of axioms ($\underline{n}\eqdef1+\ldots+1$ $n$~times) \cite[Appendix]{hume}:
			
    \begin{enumerate}
		\item $x=0\lra\forall y\:(x+y=y)$
		\item $x<y\lra\exists z\:((x+z=y)\wedge\neg(z=0))$
        \item $x=1\lra 0<x\wedge\neg\exists z\:(0<z\wedge z<x)$
		\item $x\equiv_n y\lra\exists u\:(x=\underline{n}u+y\vee y=\underline{n}u+x)$
		\item $\neg(x+1=0)$
		\item $x+z=y+z\ra x=y$
		\item $(x+y)+z=x+(y+z)$
		\item $x=0\vee\exists y\:(x=y+1)$
		\item $x+y=y+x$
		\item $x<y\vee x=y\vee y<x$
		\item $(x\equiv_n 0)\vee(x\equiv_n 1)\vee\ldots\vee(x\equiv_n\underline{n-1})$ (schema)
    \end{enumerate}

        Unlike $\PA$, where it is impossible to produce an explicit non-standard model by defining addition and multiplication recursively \cite{tennenbaum}, examples of non-stan\-dard models of $\PrA$ can be given explicitly \cite{smorynski}, \cite{lipshitz}.
	
	\begin{ex}[Recursive non-standard model]\label{ns}
		Consider all tuples $$\{(p,q)\mid p\in\QQ\ge 0,q\in\ZZ,p=0\Ra q\in\NN\}$$ with addition defined by component. This structure fulfills all the axioms of $\PrA$, yet its order type is isomorphic to $\NN+\ZZ\cdot\QQ$.
	\end{ex}

        As we will show, however, it is impossible to introduce the function $V_n$ in the model above so that it becomes a (non-standard) model of $\BA_n$.

        The values of $V_2$ (in the case of $\BA_2$, for $\BA_n$ the definition is similar) can be obtained inductively in the standard model:
        
        \begin{enumerate}
            \setcounter{enumi}{11}
            \item $V_2(x)=0\lra x=0$
            \item $\neg\exists t\:(t+t=x)\ra V_2(x)=1$
            \item $\exists t\:(t+t=x)\ra V_2(x)=V_2(t)+V_2(t)$
        \end{enumerate}

        However, axioms $(1)-(14)$ are not sufficient to describe $\BA_2$. The following conditions, all true in $\BA_2$, cannot be proved based on the axioms above, which can be shown by constructing counter-models explicitly:
        
        \begin{enumerate}
            \setcounter{enumi}{14}
            \item $\forall x\:\exists y\:(y>x\wedge V_2(y)=y)$ ``after each number, there is a power of $2$"
            \item $\forall x\:(V_2(x)=x\ra(\forall y\:(x<y<x+x\ra V_2(y)<y)))$ ``between a power of $2$ and its double, there are no more powers of $2$"
            \item $\forall x(V_2(x)=x\ra\neg\exists y(\underbrace{y+...+y}_\text{$n$ times}=x))$, $n=3,5,7,9,11,\ldots$ ``no power of $2$ is divisible by any odd number" (schema)
        \end{enumerate}        

        We suggest:
        
        \begin{conjecture}
            The axioms and schemata $(1)-(17)$ axiomatize $\BA_2$.
        \end{conjecture}
        
    \section{Structure of non-standard models}

    In this section, we will establish basic facts about non-standard models of B\"uchi arithmetic.

    First, we describe their order-type. Any model of a B\"uchi arithmetic starts with the elements $0,1,\{\underline{n}\}_{n=2}^\infty$, which we will call \emph{standard (natural) numbers}. All the remaining elements of a model will be called \emph{non-standard numbers}.

    \begin{stm}[after \cite{kemeny}]
        Any non-standard model $\Ac\models\BA_n$ has the order type $\NN+\ZZ\cdot A$, where $\langle A,<_A\rangle$ is some dense linear order without endpoints. In particular, any countable non-standard model of $\BA_n$ has the order type $\NN+\ZZ\cdot\QQ$.
    \end{stm}
    \begin{proof}
        Consider the equivalence relation: $a\sim b$ iff $|b-a|$ is a standard natural number. Each non-standard element $a$ of $\mathfrak{A}$ belongs to a subset $[a]\eqdef\{b\mid a\sim b\}$ isomorphic to $\ZZ$. We will call such fragments \emph{galaxies}. We can introduce the ordering on galaxies induced from the original ordering: $$[a]<[b]\Lra (a<b\wedge [a]\neq[b]).$$ This ordering is linear. We show its denseness and absence of endpoints by giving an example of a point to the left, right, and between any two non-standard galaxies $[a]<[b]$. For simplicity, assume $a$ and $b$ are even, because otherwise we may replace them with $a-1$ and $b-1$ respectively:

        \begin{itemize}
            \item $[a+a]>[a]$, as $a+a>a$, and $(a+a)-a=a$ is a non-standard number;
            \item $\left[\frac{a}{2}\right]<[a]$, as $a>\frac{a}{2}$, and $a-\frac{a}{2}=\frac{a}{2}$ is a non-standard number;
            \item if $a<b$, then $[a]<\left[\frac{a+b}{2}\right]<[b]$, as $a<\frac{a+b}{2}<b$, and both $\frac{a+b}{2}-a$ and $b-\frac{a+b}{2}$ are non-standard numbers.
        \end{itemize}
    \end{proof}

    Thus, for a non-standard model of $\BA_n$ its order is isomorphic to $\NN+\ZZ\cdot\QQ$. Furthermore, it is possible to introduce an induced addition on galaxies.

    \begin{lemma}
        The addition on galaxies defined as $[a]+[b]=[a+b]$ is well-defined.
    \end{lemma}
    \begin{proof}
        Let $c'=c+k$, $d'=d+l$, where $k$ and $l$ are standard numbers. Then $(c'+d')=(c+d)+k+l$ differs from $c+d$ by a standard number $k+l$ and thus belongs to the same galaxy $[c+d]$. Hence, the result of addition does not depend on the choice of the representative element in the galaxy.
    \end{proof}

    We note that it does not necessarily follow that the additive monoid of a countable non-standard model of $\BA_n$ would necessarily be isomorphic to $\NN+\ZZ\cdot\QQ$. Indeed, even though the \emph{order} structure of the galaxies is isomorphic to $\QQ_{\ge 0}$ by Cantor's theorem, this does not mean that the \emph{additive} structure of those coincides with $\QQ_{\ge 0}$.
    
    For example, the non-negative real algebraic numbers, being countable, are order-isomorphic to $\QQ_{\ge 0}$ out of necessity, but not additively isomorphic, as there is no element that generates all positive elements only through multiplying and dividing it by standard natural numbers. However, we intend to construct a model of $\BA_2$ that indeed has the monoid of galaxies isomorphic to $\QQ_{\ge 0}$.

    \section{Recovering the structure of $V_2$}

    Now we will concentrate on the B\"uchi arithmetic $\BA_2$, studying the structure of its countable non-standard models. We introduce the following definition.

    \begin{definition}
        A non-standard natural number $x$ is called a \emph{hypernumber} if $V_2(x)$ is non-standard.
    \end{definition}

    We note that:

    \begin{lemma}
        A number $x$ is a hypernumber iff it is divisible by all standard powers of $2$.
    \end{lemma}
    \begin{proof}
        Assume $x$ is not a hypernumber, that is, the value of $V_2(x)$ is standard. It is divisible by $2^k$ but not by $2^{k+1}$ for such $k$ such that $V_2(x)=2^k$. Hence, it is not divisible by all standard powers of $2$.

        Vice versa, if $x$ is not divisible by all standard powers of $2$, there is such a minimal standard $k$ that $2^k$ divides $x$ but $2^{k+1}$ does not. For each power of $2$ larger than $2^{k+1}$, it is a multiple of $2^{k+1}$ and thus also does not divide $x$. Hence, $V_2(x)=2^k$ and is standard.
    \end{proof}

    In a countable non-standard model of $\BA_2$, some non-standard numbers must be hypernumbers. Indeed, recall axiom $(15)$: $$\forall x\:\exists y\:(y>x\wedge V_2(y)=y\text{ ``after each number, there is a power of $2$"}.$$

    However, a non-standard number can be a power of $2$ (that is, be equal to its own value of $V_2$) only if it is a hypernumber.

    \begin{lemma}
        There is no more that one hypernumber in a given galaxy.
    \end{lemma}
    \begin{proof}
        Let $x$ be a hypernumber, $t$ a standard number. Then, for $k\in\NN$, it holds that $2^k|(x+t)\Lra 2^k|t$, as all $2^k$ divide $x$. Thus, $V_2(x+t)=V_2(t)$ and is standard, so $x+t$ is not a hypernumber.
    \end{proof}

    On the other hand, it is possible for a galaxy to have no hypernumbers.

    Now we can prove our earlier claim.
    
    \begin{theorem}
        In the non-standard model of Presburger arithmetic given in \sref{Example}{ns} it is not possible to set the values of $V_n$ such that that the result would be a model of $\BA_n$.
    \end{theorem}
    \begin{proof}
        Assume the contrary. Let $c$ be a non-standard power of $2$ ($V_n(c)=c$) in such a model.
        
        By construction, $c=(g,n)$ where $g\in\QQ_{>0}$ is the number of its non-standard galaxy and $n\in\ZZ$ the number of the element in the galaxy. Furthermore, $n$ cannot be equal to any number but $0$, as any pair $(g,n)$, $n\neq 0$, is divisible by $2$ only a finite number of times (until the second element of the pair is odd).

        But, on the other hand, any pair $(g,0)$, $g\in\QQ_{>0}$, is divisible, say, by $3$, which is impossible for a power of $2$ by a theorem of $\BA_n$.
    \end{proof}

    According to the facts expressible by first-order formulas of $\BA_2$, in a model of $\BA_2$ there should be a chain of powers of $2$ such that there are no further powers of $2$ between a given power and the double of it. Let us fix a value of $c$ such that $c$ is non-standard and $V_2(c)=c$.
    
    We will introduce the names for the galaxies as follows. The galaxy $[c]$ will be denoted $\mathbf{c}$. As $c$ can be freely multiplied by any natural $a$ and divided by $2$ any number of times, we may naturally define $\mathbf{ac}:=[ac],a\in\NN$ and $\mathbf{c/{2^k}}:=[c/{2^k}]$, $k\in\NN$. Furthermore, $\mathbf{ac/{2^k}}:=[ac/{2^k}]$ when $a$ is odd, $k\in\NN$.

    As a power of $2$, $c$ is not divisible by any odd number. But, for each odd $b$, (exactly) one of the numbers $c-b+1,\ldots,c-1,c$ is divisible by $b$. Let us define $\mathbf{c/b}=[(c-t)/b]$ where $t\in\{0,\ldots,b-1\}$ is the only such $t$  that $b|(c-t)$.

    Now, finally, for an arbitrary positive rational $2^k (a/b)$, $a,b$ are odd, $k\in\ZZ$, we define $\mathbf{(a/b)2^k c}=[(2^k a(c-t))/b]$ where $t$ is as above.

    The previous definitions may be summarized as follows: a number $x$ is said to belong to a galaxy $\mathbf{(a/b)c}$ as long as numbers $\underline{b}x$ and $\underline{a}c$ belong to the same galaxy. This implies:
    
    \begin{lemma}\label{agre}
        The induced addition on the galaxies $\mathbf{\frac{a}{b}c}$ agrees with the names of the galaxies.
    \end{lemma}
    
    
    
    Now we establish the values of $V_2$ on the elements of the galaxies $\mathbf{(a/b)c}$.

    According to the theorems of $V_2$, $V_2(\underline{2}m)=2 V_2(m)$ and $V_2(\underline{a}m)=V_2(m)$ whenever $a$ is odd.
    
    This means that, for each natural $a 2^k$ such that $a$ is odd, $$V_2(a 2^k\cdot c)=2^k V_2(c)=2^k c.$$

    Similarly, $V_2(a\cdot c/{2^k})=V_2(c)/{2^k}=c/{2^k}$.

    For any remaining element $a 2^k c+t$ of a galaxy $\mathbf{a 2^k c}$, $k\in\ZZ$, $t$ is a standard integer, we obtain $V_2(a 2^k c+t)=V_2(t)$, a standard value.
    
    Now we calculate the values of $V_2(x)$ where $x$ is in a galaxy of the form $\mathbf{\frac{a}{b}c}$ where $\frac{a}{b}$ is not binary rational. In order to do that, we need to fix a particular $t\in\{n-1,\ldots,0\}$ such that $c-t$ is divisible by $n$, for each odd $n\ge 3$.
    
    For simplicity of calculation, let us define $c$ to be divisible by all odd numbers with the remainder $1$, that is, make $c-1$ divisible by all odd standard numbers (and by no even number). After that, we can calculate the remainders of the two-sided sequence $\ldots,c/4,c/2,c,2c,4c,\ldots$ inductively, as, according to Euler's theorem, remainders of $2^k$ modulo each odd $n$ form a predictable cycle of the length dividing $\varphi(n)$.
    
    Now, for example, for the galaxy $\mathbf{c/3}$, we assume $3|(c-1)$. Thus, $V_2((c-1)/3)=V_2(c-1)=1$, as $V_2$ does not change after division by an odd number, and $(c-1)$ is odd. Furthermore, for all $t\in\ZZ$, $$V_2((c-1)/3+t)=V_2((c+3t-1)/3)=V_2(c+3t-1)=V_2(3t-1).$$ In particular, this means all the values of $V_2$ are finite.

    We can overview the computations above as follows:

    \begin{theorem}
        For a number $x$ belonging to a galaxy $\mathbf{(a/b)2^k\cdot c}$, $a\in\ZZ$, $b\in\NN$, $a, b$ are odd, $k\in\ZZ$, the following holds ($d\in\ZZ$):

        \begin{enumerate}
            \item If $b=1$, $V_2(a 2^k\cdot c)=2^k c$, otherwise $V_2(a 2^k\cdot c+d)=V_2(d)$;
            \item If $b\ge 3$, $V_2(a 2^k\cdot (c-1)/b+d)=V_2(bd-a 2^k)$, when $k\ge 0$, $V_2(a 2^k\cdot (c-1)/b+d)=V_2(bd 2^{-k}-a)$, when $k<0$.
        \end{enumerate}
    \end{theorem}
    \begin{proof}
        Point $(1)$ is already shown. For point $(2)$, $$V_2\left(\frac{a 2^k(c-1)}{b}+d\right)=V_2\left(\frac{a 2^k c-a 2^k+bd}{b}\right)=V_2(a 2^k c-a 2^k+bd).$$ Now, if $k\ge 0$, the expression above simplifies to $V_2(bd-a 2^k)$, otherwise we need to multiply by $2^{-k}$ first to make the summands integer: $$V_2(a 2^k c-a 2^k+bd)=V_2(a c-a+bd 2^{-k})=V_2(bd 2^{-k}-a).$$
    \end{proof}

    By the argumentation above, we have established:

    \begin{lemma}
        Beside the binary rational multiples of $c$, there are no further hypernumbers in the galaxies of the form $\mathbf{(a/b)c}$.
    \end{lemma}

    Finally, we notice that the union of the standard galaxy and all the galaxies of the form $\mathbf{a}{b}c$ is closed under addition, so we may limit ourselves to considering only the galaxies of the type, deleting any other galaxies, if they exist. Note that the additive structure on the galaxies of the remaining model turns out to be exactly isomorphic to $\QQ_{\ge 0}$ by \sref{Lemma}{agre}.

    Checking axioms $(1)-(11)$ of $\PrA$ on the resulting structure is routine. As an example of axiom schema $(11)$, out of the three sequential numbers in the galaxy $\mathbf{2c/5}$, namely, $2(c-1)/5+3$, $2(c-1)/5+4$, and $2(c-1)/5+5$, exactly one is divisible by $3$. It is $2(c-1)/5+3$, as both $2(c-1)$ and $3$ are divisible by $3$. The result of the division is $2(c-1)/15+1$. The properties of $V_2$ given by axioms $(12)-(14)$ are also confirmed by definition.
    
    \printbibliography[heading=bibintoc,title={References}]

@article{zap2022,
	title={On Interpretations of Presburger Arithmetic in Büchi Arithmetics},
	author={Zapryagaev, A. A.},
	journal={Doklady Mathematics},
	year={2023},
	volume={107},
        pages={89--92},
	doi={10.1134/S1064562423700655}
}

@mastersthesis{bruyere,
	title={Entiers et automates finis},
	author={Bruy{\`e}re, V.},
	school={Universit\'e de Mons},
	type={M\'emoire de fin d'\'etudes},
	year={1985}
}

@article{bv,
	title={Logic and p-recognizable sets of integers},
	author={Bruy\`ere, V. and Hansel, G. and Michaux, C. and Villemaire, R.},
	journal={Bulletin of the Belgian Mathematical Society Simon Stevin},
	volume={1},
	number={2},
	pages={191--238},
	year={1994},
	doi={10.36045/bbms/1103408547}
}

@article{kemeny,
	title={Undecidable problems of elementary number theory},
	author={Kemeny, J. G.},
	journal={Mathematische Annalen},
	volume={135},
	pages={160--169},
	year={1958}
}

@inproceedings{kn,
	title={Automatic presentations of structures},
	author={Khoussainov, B. and Nerode, A.},
	booktitle={International Workshop, LCC '94, Indianapolis, IN, USA, October 13--16, 1994. Selected Papers},
	pages={367--392},
	year={2005},
	publisher={Springer},
	address={Berlin, Heidelberg},
	editor={Leivant, D.},
	series={Lecture Notes in Computer Science},
	volume={960},
	doi={10.1007/3-540-60178-3_93}
}

@book{smorynski,
	title={Logical number theory I: An introduction},
	author={Smory{\'n}ski, C.},
	year={1991},
	publisher={Springer-Verlag},
	address={Berlin, Heidelberg},
	isbn={9783642754623},
	doi={10.1007/978-3-642-75462-3}
}

@article{tennenbaum,
	title={Non-archimedean models for arithmetic},
	author={Tennenbaum, S.},
	journal={Notices of the American Mathematical Society},
	volume={6},
	number={3},
	pages={270},
	year={1959}
}

@article{hume,
	title={Hume’s principle, beginnings},
	author={Visser, A.},
	journal={The Review of Symbolic Logic},
	volume={4},
	number={1},
	pages={114--129},
	year={2011},
	doi={10.1017/S1755020310000316}
}

@article{lipshitz,
  title={The additive structure of models of arithmetic},
  author={Lipshitz, L. and Nadel, M.},
  journal={Proceedings of the American Mathematical Society},
  volume={68},
  number={3},
  pages={331--336},
  year={1978},
  doi={10.1090/S0002-9939-1978-0491158-5}
}

@phdthesis{zapthe,
    author = {Zapryagaev, A. A.},
    title = {Interpretations in weak arithmetical theories},
    school = {NRU Higher School of Economics},
    year = {2023},
    url = {https://www.hse.ru/data/xf/026/737/2034/Диссертация.pdf}
}

@article{vf,
  title={Friedman-reflexivity},
  author={Visser, A.},
  journal={Annals of Pure and Applied Logic},
  volume={173},
  number={9},
  pages={103--160},
  year={2022},
  doi={10.1016/j.apal.2022.103160}
}
\end{document}